\documentclass[11pt]{amsart}

\usepackage{mathrsfs}
\usepackage{amssymb}
\usepackage{amsthm}
\usepackage{dsfont}
\usepackage{amssymb}
\usepackage{amsmath,amscd}
\usepackage[all,cmtip]{xy}
\usepackage{mathrsfs}
\usepackage{multirow}
\usepackage{hyperref}
\usepackage{graphicx}
\usepackage{amsfonts, amsthm}
\usepackage{fullpage}
\usepackage{enumerate}
\usepackage{appendix}
\usepackage{cite}

\normalfont\upshape

\theoremstyle{plain}
\newtheorem{thm}[subsection]{Theorem}
\newtheorem{lemma}[subsection]{Lemma}
\newtheorem{prop}[subsection]{Proposition}
\newtheorem{cor}[subsection]{Corollary}

\theoremstyle{definition}
\newtheorem{rmk}[subsection]{Remark}
\newtheorem{defn}[subsection]{Definition}

\let\a\alpha
\let\b\beta
\let\d\delta

\let\e\epsilon

\let\g\gamma

\let\i\iota

\let\l\lambda

\let\p\psi
\let\o\omega

\let\s\sigma

\let\th\theta

\let\x\xi

\def\scr{\mathscr}
\def\cal{\mathcal}

\let\G\Gamma

\let\O\Omega

\newcommand{\ra}{\longrightarrow}
\newcommand{\xa}{\xrightarrow}

\newcommand{\bGn}{ G[p^n]\otimes_{\Z} H[p^n]}

\newcommand{\im}{{\rm im}\:}

\newcommand{\DD}{{\mathbb D}}

\newcommand{\F}{{\mathbb F}}
\newcommand{\Z}{{\mathbb Z}}
\newcommand{\C}{{\mathbb C}}
\newcommand{\Q}{{\mathbb Q}}
\newcommand{\R}{{\mathbb R}}

\newcommand{\cO}{\mathcal{O}}
\newcommand{\cT}{\mathcal{T}}
\newcommand{\cV}{\mathcal{V}}

\newcommand{\sL}{{\mathscr L}}

\newcommand{\Spec}{{\mbox{Spec }}}

\newcommand{\End}{{\mathrm{End}}}

\newcommand{\cri}{{\mathrm{cris}}}
\newcommand{\Def}{{\mathrm{Def}}}
\newcommand{\Aut}{{\mathrm{Aut}}}
\newcommand{\et}{{\mathrm{et}}}
\newcommand{\mult}{{\mathrm{mult}}}

\newcommand{\colim}{{\mathrm{colim}}}

\newcommand{\Ht}{\mathrm{ht}}

\newcommand{\Hom}{{\mathrm{Hom}}}

\begin{document}

\title[On the deformation of a Barsotti-Tate group] {On the deformation of a Barsotti-Tate group over a curve}
\author{Jie Xia }\address{ Mathematics Department, Columbia University in the city of New York} \email{xiajie@math.columbia.edu}
\maketitle

\begin{abstract}
In this paper, we study deformations of pairs $(C,G)$ where $G$ is a height 2 BT( or BT$_n$) group over a complete curve on algebraically closed field $k$ of characteristic $p$. We prove that, if the curve $C$ is a versal deformation of $G$, then there exists a unique lifting of the pair to the Witt ring $W$. We apply this result in the case of Shimura curves to obtain a lifting criterion. 
\end{abstract}

\section{Introduction}
Illusie, along with Grothendieck and Raynaud, develop the deformation theory of BT groups (see \cite{Ill}). One result in \cite{Ill} is that, roughly speaking,  BT groups are unobstructed over an affine base. In our paper, rather than an affine base, we consider the deformation of the pair of a complete curve and a BT group of height 2  over this curve. We show that there is no obstruction to lift the pair $(C,G)$ from characteristic $p$ to characteristic $0$ when the curve $C$ is a versal deformation. Furthermore, \ref{lifting in general pol} indicates that this deformation result is a main ingredient in lifting a curve in characteristic $p$ to a Shimura curve. This leads to our work on how to define Shimura varieties in positive characteristic. The relevant results will appear in our upcoming paper. 

We note that Kang Zuo and his collaborators also work with the special rank 2 bundles (associated to height 2 BT groups via Dieudonne functor). For example, in \cite{Zuo3}, amid other results, they discover the periodicity  of rank 2 semistable Higgs bundle.  Shinichi Mochizuki has also considered a similar problem. In \cite{Mochizuki}, he developed the theory of the rank 2 indigenous bundle, which is closely related to the height 2 BT group. 

\subsection{Main results}In this paper, let $k$ be an algebraically closed field of characteristic $p$, $W$ Witt ring $W(k)$, $W_n$ the truncated ring $W(k)/p^{n+1}$ and $C$ a smooth proper curve of genus $\geq 2$ over $k$. Let $G$ be a Bassotti-Tate(BT) group of height 2 over $C$ . For the definition of BT or truncated BT groups, we refer to Chapter 1 of \cite{Messing}.  

In this situation,  (\cite[A.2.3.6]{Ill} ) says that the following two conditions are equivalent.
\begin{itemize}
 \item For any closed point $c\in C$, let $\hat{C}_c$ be the completion of $C$ over $c$. The restriction of $G$ over $\hat{C}_c$ is the versal deformation of $G_c$ in the category of local artinnian $k-$algebras with residue field $k$. 
 
 \item the Kodaira-Spencer map: 
\begin{equation} \label{ks equation}
\rm{ks}: T_C \rightarrow t_G \otimes t_{G^*} \end{equation} is an isomorphism.
\end{itemize}
Here $t_G$ is the degree  0 cohomology of the dual cotangent complex of $G$, i.e. $t_G=H^0(\check{l}_G)$. If the pair $(C, G)$ satisfies either of the conditions, then we say that $C$ is a versal deformation of $G$. 

\begin{thm} \label{main theorem}
If $C$ is a versal deformation of $G$, then there exists a unique curve $C'$ over $W$ which is a lifting of $C$ and admits a lifting $G'$ of $G$, 
\[\xymatrix{
G \ar[r] \ar[d] & G' \ar[d]\\
C\ar[r] \ar[d] &C' \ar[d]\\
k \ar[r]& W.
} \]  Furthermore, $G'$ is unique up to a unique isomorphism. 
\end{thm}

\begin{rmk}
Let $C_n$ be any flat curve over $W_n$ with special fiber $C$. It is clear that $C_n$ is proper and smooth, for any $n\geq 1$ and so is $C'$. 
\end{rmk}

\begin{rmk}
We also have a similar result for truncated BT groups but the uniqueness is not guaranteed. For details, see \ref{the truncated case} .
\end{rmk}

As an application of \ref{main theorem}, we show a liftability result on Shimura curves of Mumford type in  \ref{lifting in general pol}.
In \cite{Mum}, Mumford defines a family of Shimura curves of Hodge type over a number field which we call Shimura curves of Mumford type. If  $X \ra C$ is a family of abelian varieties over $k$ with maximal Higgs field and its associated $p-$divisible group $X[p^\infty]$ has a special decomposition, then $X \ra C$ is a good reduction of a Shimura curve of Mumford type. The proof is technical yet the strategy is fairly simple. Firstly \ref{main theorem} implies that $X \ra C$ can be lifted to a formal family of abelian varieties over $W$. Then we prove any polarization lifts as well, which implies that the formal family is algebraizable. 

\subsection{Structure of the paper}
Theorem \ref{main theorem} is proved in Sections \ref{1st order} and \ref{higher order deformation} in the following steps:
\begin{enumerate} [I]
\item Compute the obstruction class of the lifting to $W_2$.
\item First order deformation: a choice of a lifting of $C$ to $W_2$ can kill the obstruction class.
\item Higher order deformation: recursively apply steps I and II to $W_n$ for any $n$ to obtain a formal scheme over $\mathrm{Spf}(W)$.
\item Algebraization: apply Grothendieck's existence theorem to algebraize the formal scheme.
\end{enumerate}

In Section \ref{equicharacteristic} and \ref{trivial KS map}, we give two variations of \ref{main theorem} which we will use in our future work.  In Section \ref{application}, we mainly prove \ref{lifting in general pol}.

\subsection{Notation}
We follow the notation in \cite{Ill}. For any BT or truncated BT group $G$ over $C$, let $e: C \ra G[p]$ be the identity element. Let $l_G=Le^*L_{G/C}$ be the pull back of the cotangent complex of $G/C$. Let $\o_G=e^* \O^1_{G[p]/C}$ and $t_G=H^0(\check{l}_G)$. The bundle $\DD(G)_C=\DD(G[p])$ admits a Hodge filtration (see \cite[A 2.3.1]{Ill})
\[0\ra \o_G \ra \DD(G)_C \ra t_{G^*} \ra 0.\]
Note that $\DD(G)$ admits a connection $\nabla: \DD(G)_C \ra \DD(G)_C \otimes \O^1_C$, which induces the Higgs field $\theta_G: \o_G \ra t_{G^*} \otimes \O^1_C$. We say the Higgs field is maximal if it is an isomorphism. Also if $G$ has height 2, $\o_G$ and $t_{G^*}$ has rank 1, then the Higgs field gives the Kodaira-Spencer map $T_C \ra t_G \otimes t_{G^*}$. 

For a family of abelian varieties $X\ra C$, we use $\o_X, t_X$ to represent the corresponding sheaves associated to the $p-$divisible group $X[p^\infty]$. 

\textbf{Acknowledgments.} I would like to express my deep gratitude to my advisor A.J. de Jong, for suggesting the project to me and for his patience and guidance throughout its resolution. This paper would not exist without many inspiring discussions with him.  I also want to thank Professor Kang Zuo and Mao Sheng for the helpful conversations. Lastly, I thank Vivek Pal for checking the grammatical errors.

\section{An example of the versal deformation} \label{example}

We construct the example from a fake elliptic curve, using Morita's equivalence.

Let $D$ be a nonsplit quaternion algebra over $\Q$ and $\cal{O}\subset D$ be a maximal order. The fake elliptic curve is a smooth proper Shimura curve of PEL type in the moduli $\cal{A}_{2,1}$ of abelian surfaces, defined by the following: 
\begin{enumerate}
\item
$A/\C$ is an abelian surface, 
\item $D\hookrightarrow \End(A)\otimes \Q$ such that $H_1(A,\Z)$ as an $\cal{O}-$module is isomorphic to $\cal{O}$,
\item the Rosati involution associated to the polarization $\scr{L}$ induces the involution on $\cal{O}$,
\item a level $n$ structure.  
\end{enumerate}

Since every PEL type Shimura variety has an integral model(see \cite{Kisin}), we can choose a prime $\frak{p}$ such that $\frak{p}$ is prime to the discriminant  $\text{disc}(D)$ and the fake elliptic curve admits a good reduction over $\frak{p}$. We denote the reduction as $S$. Then $S$ is proper, smooth and admits a family of abelian surfaces $Y$, and hence a family $H$ of BT groups of height 4 over $S$. Furthermore, $M_2(\F_p) \cong D \otimes_\Z \F_p$ acts on $H$. 

By \cite[Theorem 0.5]{Moller}, the PEL type Shimura curves all have maximal Higgs field. 

By Morita's equivalence, the following 
\[
\begin{aligned}\{\text{BT groups of height 4 with an action of }M_2(\F_p)\}\leftrightarrow&\{\text{BT groups of height 2}\}\\
H \mapsto& \begin{pmatrix}
  1 & 0 \\
  0 & 0 \\
 \end{pmatrix}H\\
H'\times H' \gets& H'
\end{aligned}\] is an equivalence of categories. So we have a family $H'$ of BT group of height 2 over $S$. 

Furthermore, since the base change preserves the maximal Higgs field, by \cite[A 2.3.6]{Ill}, $H' \ra S$ is a versal deformation of height 2 BT groups.

\section{First order deformation} \label{1st order}

In this section, without specific indications, $G$ is a BT or BT$_n$ group.

\subsection{Step I} First recall the following theorems by Grothendieck and Illusie:

\begin{prop}(\cite[Theorem 4.4]{Ill}) \label{locally unobstructed}
Let $S_0 \ra S \stackrel{i}{\ra} S'$ be a closed immersion defined by ideal sheaf $J\subset K$ with $JK=0$ and $p^N \cal{O}_{S_0}=0$. Let $G$ be a BT$_n$ group over $S$ with $n\geq N$ or a BT group. Assume $S'$ is affine and $p$ is nilpotent on $S'$. Then
\begin{enumerate}
\item there exists a $BT_n$ group $G'$ on $S'$ as a deformation of $G$,
\item the set of deformations up to isomorphism is a torsor under the group $(t_G \otimes t_{G^*} \otimes J)(S_0)$
\item if $n-k>N$, then the morphism 
\[\Def(G,i) \ra \Def(G[p^k],i)\] is bijective. 
\end{enumerate} 
\end{prop}

\begin{prop}  (\cite[Corollary 4.7]{Ill}) \label{reduce to truncated 0}
Assumption as \ref{locally unobstructed}. Let $H$ be a BT group on $S$.
\begin{enumerate} 
\item If $n\geq N$, then the map 
\[\Def(H, i) \ra \Def(H[p^n],i)\] 
is bijective.
\item The automorphism group of the deformation $H'$ of $H$ to $S'$ is trivial.
\end{enumerate}
\end{prop}

\begin{cor} \label{reduce to truncated}
Let $S_0 \ra S \stackrel{i}{\ra} S'$ be a closed immersion defined by ideal sheaf $J\subset K$ with $JK=0$ and $p^N \cal{O}_{S_0}=0$. Let $G$ be a BT$_n$ group over $S$ with $n\geq N$ or a BT group. Assume $p$ is nilpotent on $S'$ and $S_0$ is a proper smooth curve over an algebraically closed field $k$ of characteristic $p$. If $G[p] $ can be lifted to $S'$, then $G $ can also be lifted to $S'$. Hence in this case, the obstruction to lift $ G$ is the same as that of $G[p]$. 
\end{cor}

\begin{proof}
By \ref{reduce to truncated 0} (1), we know that over any affine open $U\subset S$, $G_U$ can be uniquely lifted to $U' \subset S'$ which is compatible with the lifting of $G[1]_U$. Over $U \cap V$, $G_{U'}$ and $G_{V'}$ both induce the same lifting of $G[p]_{U\cap V}$. Hence due to the bijection in \ref{reduce to truncated 0}(1), $G_{U'} \cong G_{V'}$.  Since $S_0$ is a curve, there is no need to check the coboundry condition. Hence $\{G_{U'}\}$ glue to a global BT group over $S$. 

By \ref{locally unobstructed} (3), the same argument works for the truncated case. 
\end{proof}

\begin{rmk}
From \ref{reduce to truncated}, we only know $G$ is liftable but it may not be globally compatible with the lifting of $G[p]$.
\end{rmk}

\begin{cor} \label{torsor of BT}
Assumptions as \ref{reduce to truncated} and further assume $G$ is a BT group. The deformation space $\Def(G,i)$ is a torsor under $H^0(S_0, t_G\otimes t_{G^*}\otimes J)$. 
\end{cor}

\begin{proof}
Let $\{U_i\}$ be an affine open cover of $S$. Let $\{U'_i\}$ be the lifting of $\{U_i\}$ to $S'$ and $\{U_{i0}\}$ be the corresponding affine open cover on $S_0$. 

From (\cite[3.2(b)]{Ill}), the deformation space $\Def_{S'}(G)$ (up to isomorphism) is a torsor under some $k-$vector space $V$. By \ref{locally unobstructed},    for any $v\in V$ and any $i$, there exists $s_i \in (t_G \otimes t_{G^*} \otimes J)(U_{i0})$ such that $v_{|U'_{i}}=s_i$. Over each $U_{ij}$, 
\[{s_i}_{|U_{ij0}}=v_{|U'_{ij}}={s_j}_{|U_{ij0}}.\] So $\{s_i\}$ patch to a global section $s$ of $t_G \otimes t_{G^*} \otimes J$. Hence we have the linear transformation 
\begin{equation}\begin{aligned} \label{surj between torsors} 
V &\ra H^0(S_0, t_G \otimes t_{G^*}\otimes J)\\
v & \mapsto  s
\end{aligned}
\end{equation} It is easy to show this transformation is surjective. 

The kernel of morphism (\ref{surj between torsors}) consists of  the elements which induce locally isomorphisms between any two deformations of $G$. Since $G$ is a BT group,  by \ref{reduce to truncated 0} (2), over each $U_i$, the isomorphism between two deformations $G_1$ and $G_2$ which induces the identity on $G_{|U_i}$ is unique. Hence the local isomorphisms over each $U_i$ can glue to a global isomorphism. Therefore the kernel is trivial. The space $\Def(G, i)$ is a torsor under $H^0(S_0, t_G \otimes t_{G^*}\otimes J)$
\end{proof}

Now we can show 
\begin{thm} \label{obstruction and torsor}
Assumptions as \ref{main theorem}, for any deformation $C_2$ of $C$ over $W_2$, \begin{enumerate}
\item the obstruction class $\rm{ob}_{C_2}(G)$ of the deformation of $G$ to $C_2$ is in $H^1(C, T_C)$, 
\item if the obstruction class vanishes, then the deformation space is a torsor under $H^0(C, T_C)$.
\end{enumerate}
\end{thm}

\begin{proof}
(1) By \ref{reduce to truncated 0} (2) and \ref{reduce to truncated}, we only need to consider the truncated BT group $G$. Choose any open affine cover $\{U_i\}$ of $C$ and the deformation $U'_i$ of $U_i$ to $W_2$ is unique up to isomorphisms. By \ref{locally unobstructed}, $G$ is locally liftable and the set of deformations of $G_{U_i}$ (up to isomorphisms) is a torsor under the group $(t_G \otimes t_{G^*} \otimes J)(U_i)$. Since the thickening just has the first order, $J\cong k$ and the set of lifting is a torsor under $(t_G\otimes t_{G*})(U_i)$. On $U_{{ij}}=U_i \cap U_j$, the restriction from $G_{U'_i}$ and $G_{U'_j}$ give two deformations of $G_{U_{{ij}}}$. 

By \ref{locally unobstructed}, we can choose $s_{ij}$ the element in $(t_G\otimes t_{G^*})(U_{ij})$ sending $[G_{U'_{ij}}]$ to $[G'_{U'_{ij}}]$, i.e. \[s_{{ij}}=[G_{U'_{ij}}]-[G_{U'_{ji}}] \in (t_G\otimes t_{G^*})(U_{ij}).\] 
The vanishing of the Cech cocycle $\{s_{ij}\}$ precisely gives a deformation of $G$. Note by \ref{reduce to truncated}, we can adjust the deformation on $U_i$ by elements in $(t_G\otimes t_{G^*}) (U_i)$.  Hence the obstruction is in $H^1(C, t_G \otimes t_{G^*})$. 

Since $C$ is a versal deformation, the Kodaira-Spencer map gives 
\[ t_G\otimes t_{G^*} \cong  T_C.\]

(2) It directly follows from \ref{torsor of BT}. 
\end{proof}

\begin{rmk}
From the proof, we know \ref{obstruction and torsor} (1) also holds for truncated BT groups.
\end{rmk}

\begin{rmk} \label{the uniqueness remark}
Notation as \ref{main theorem}. If we assume $g(C)\geq 2$, then $H^0(C, T_C)=0$ and hence for the BT group $G$, there is a unique (up to a unique isomorphism, \ref{reduce to truncated 0}(2)) deformation $G'$ to $C'$, if exists. 

And according to \cite[Corollary 5.4]{Jun}, if a curve along with a family of abelian varieties is liftable, preserving the maximal Higgs field, then the genus of the curve is necessarily greater than 2. 
\end{rmk}

\subsection{Step II}
In this part, we prove \ref{KS}. 

Firstly, we  recall that the deformation space of $C$ to $W_2$ is $H^1(C, T_C)$. Without loss of generality, assume $C$ can be covered by two affine open $U$ and $V$,with the unique  deformations $U'$ and $V'$ to $W_2$. 

Fix the embedding $(U\cap V )' \longrightarrow V'$ and alternating the embedding $(U\cap V )' \ra U'$ gives a different lifting of $C$. And all the deformations of $C$ can be obtained in this way.

\begin{equation} \label{the geometric diagram}
\xymatrix{
U' & (U\cap V)' \ar@<2pt>[l] \ar@<-2pt>[l] \ar[r] & V'\\   
U \ar[u] & U\cap V \ar[u] \ar[l] \ar[r] & V \ar[u] 
}\end{equation}

Since all the schemes involved in \ref{the geometric diagram} are affine, we can interpret the diagram in the level of commutative algebras: 

\[\xymatrix{
B' \ar@<2pt>[r]^{\p_1} \ar@<-2pt>[r]_{\p_2} \ar[d]_{/p}& A' \ar[d]_{/p} & C' \ar[l] \ar[d]_{/p} \\ 
B \ar[r] & A & C \ar[l]
}\]
where $\p_1-\p_2=p\d$ where $\d \in \rm{Der}(A)$.  Here $/p$ means quotient of $p$.

Let $G_{U'}$ be any deformation of $G_U$. Then the map $\p_1$ and $\p_2$ induce two deformations $G_{U'}\otimes_{\p_1} A'$ and $G_{U'}\otimes_{\p_2}A'$ of $G_{U\cap V}$. By \ref{locally unobstructed} (2), we can denote the element in $t_G \otimes t_{G^*} (U \cap V)$ taking $G_{U'}\otimes_{\psi_2} A' $ to $G_{U'} \otimes_{\psi_1} A' $ as $[G_{U'}\otimes_{\p_1}A']-[G_{U'}\otimes_{\p_2} A'].$

\begin{thm} \label{KS} Restricting the $\rm{ks}$ to $U\cap V$(for the definition of ks, see the morphism (\ref{ks equation})), we have
\[\rm{ks}(\d)=[G_{U'}\otimes_{\p_1}A']-[G_{U'}\otimes_{\p_2} A']. \] 
\end{thm}

\begin{proof}

By (\cite[4.8.1]{Ill}), we have 
\begin{equation}  \label{ks equation 2}
\rm{ks}(\d|_x)=[f^*G]-[G_{|x} \otimes_k k[\e]] \end{equation}
where $\e^2=0$ and $f:k[\e] \ra C$ is induced from the tangent vector $\d_{|x}$. Let $f(\e)=s \,(\text{mod } \frak{m}^2_x)$ where $\frak{m}_x$ is the maximal ideal of $x$ in $A$ and $s$ is actually the generator of $\frak{m}_x/\frak{m}^2_x$. Since $A_x$ is a UFD of dimension 1, $A_x$ is a PID and thus $\frak{m}_s$ is generated by $s$. 

For any $x\in U\cap V$, by Hensel's lemma, there exists a section $x'$  of $C' \ra W_2$  lifting the point $x$. We can choose the ideal $I_{x'}$ such that the $x'$ is given by $A'\ra A'/I_{x'}$ locally on $U\cap V$. 

Based on (\ref{ks equation 2}) it suffices to show 
\[[G_{U'}\otimes_{\p_1}A'/I_{x'}]-[G_{U'}\otimes_{\p_2} A'/I_{x'}]=[f^*G]-[G_{|x} \otimes_k k[\e]].\]
We can prove the result by local computation. 
\begin{equation} \label{the diagram}
\xymatrix{
B \ar@{-->}[r] \ar[d]^{.p} & A \ar[d]^{.p} \ar[r]^{/\frak{m}_x}  & k \ar[d]^{.p} \\
B' \ar[d] \ar@<2pt>[r]^{\p_1} \ar@<-2pt>[r]_{\p_2} & A' \ar[d] \ar[r]^{/I_{x'}} & W_2\ar[d]\\
B \ar[r] & A \ar[r] & k
} \end{equation} where $\p_2-\p_1=p\d$. It is easy to see that the top arrow $A\ra k$ is the quotient by the maximal ideal $\frak{m}_x$ of $x$. 

Let  $I_1=(\p_1)^{-1}(I_{x'})$ and $I_2=(\p_2)^{-1}(I_{x'})=(\p_1+p\d)^{-1}(I_{x'})$. Note $U'=\Spec B'$ and then 
\[G_1:=G_{U'} \otimes {B'/I_1}, G_2:=G_{U'}\otimes{B'/I_2}\] 
are both deformations of $G_U \otimes k$ to $W_2$ with
\[[G_{U'}\otimes_{\p_1}A'/I_{x'}]-[G_{U'}\otimes_{\p_2} A'/I_{x'}]=[G_2 ]-[G_1].\]

Since $s\in I_1/p=\frak{m}_x$, we can choose a lifting $\tilde s$ of $s$ to $B'$ such that $\tilde s \in I_1$. Then $(\p_1+p\d)(\tilde s - p)=0$, i.e. $\tilde s-p \in I_2$. Since $\frak{m}_x=(s)$, we have 
\[I_1=(\tilde s), I_2=(\tilde s-p).\]

Let $I=(\tilde s^2, ps)$. Note $I\subset I_1 \cap I_2$ and since $U'$ is an affine flat scheme over $W_2$, there exists an injection 
\begin{equation}\label{the i} i: W_2 \hookrightarrow B'/I.\end{equation}

Note $B'/(\tilde s, p)=k$ and there exists a section $k \ra B'/(\tilde s^2, p)$. Thus $B'/(\tilde s ^2,p )=k[\e]/\e^2$. We have the following diagram: 
\begin{equation}  \label{the local deformation}
\xymatrix{
&k[\e] \ar[dl]\\
k&&B'/I \ar[ul]_{\phi_3} \ar@<2pt>[dl]^{\phi_2} \ar@<-2pt>[dl]_{\phi_1} \ar[ll]_J\\
&W_2 \ar[ul]
}\end{equation} where $\phi_1(\tilde s)=0, \phi_2(\tilde s)=p, \phi_3(p)=0$ and $(B'/I)/J=k$.  The ideal $J$ is generated by $(\tilde s, p)$.

 We consider the triple 
\[\{G_{U'} \otimes B'/I, G_1 \otimes_{W_2} B'/I, G_2 \otimes_{W_2} B'/I\}\] 
of deformations of $G_U \otimes k$ on $B'/I$ where $G_i \otimes_{W_2} B'/I$ is through $i$ (see \ref{the i}). Let
\[\begin{aligned}
& \xi_1&=&[G_{U'} \otimes B'/I]-[ G_1 \otimes B'/I] &\\
&\xi_2&=&[G_{U'} \otimes B'/I]-[ G_2 \otimes B'/I] &\\
&\xi_3&=&[ G_2 \otimes B'/I]-[ G[p] \otimes B'/I]&
\end{aligned} \]where $\xi_i \in t_{G_x} \otimes t_{G^*_x} \otimes_k J $(prop. \ref{locally unobstructed}). 

Then $\phi_i$ induces 
\[\begin{aligned}
\phi^*_i: t_G\otimes t_{G^*}\otimes_k J \ra & t_G \otimes t_{G^*} \otimes_k pW_2, i\in\{1,2\}.\\
\phi^*_3: t_G \otimes t_{G^*} \otimes_k J \ra & t_G \otimes t_{G^*} \otimes_k k[\e].
\end{aligned}\]

Since $G$ has height 2, $t_{G_x}$ is a one dimensional space over $k$. Hence $t_{G_x}\otimes t_{G^*_x}\cong k$.

Choose the canonical basis $(\e,p),(p),(\e)$ for $J, pW_2, \e k[\e]$, respectively. Then we can write 
\[\phi^*_i: k^2 \ra k, 1\leq i \leq 3. \] Then for any $a,b\in k$

\[\begin{aligned}
\phi^*_1(a\e+bp)=&b,\\
\phi^*_2(a\e+bp)=&a+b,\\
\phi^*_3(a\e+bp)=&a.
\end{aligned}\]In particular, 
\begin{equation} \label{key equation} \phi^*_2=\phi^*_1+\phi^*_3.\end{equation}

Then we have the following relations: \begin{enumerate}
\item $\phi^*_1(\x_1)=0$, $\xi_1=\x_2+\x_3$, 
\item $\phi^*_2(\x_2)=0$, $\phi^*_3(\x_2)=[G_U \otimes_f k[\e]]-[G_U \otimes_k k[\e]]$,
\item $\phi^*_1(\x_3)=\phi^*_2(\x_3)=[G_2 ]-[G[p]].$
\end{enumerate}

For (1), the pull back of the pair $\{G_{U'} \otimes B'/I, G[p] \otimes B'/I\}$ through $\phi_1$ is identically $G[p]$. For (2), $\phi^*_3(B'/I)=B/(s^2)$ and then $\phi^*_3(G_{U'} \otimes B'/I)=G_U \otimes_f k[\e]$. Similarly, one can verify the other formulas.

Hence 
\begin{equation} \label{the final equation} 0=\phi^*_1(\x_1)=(\phi^*_2-\phi^*_3)(\x_2+\x_3)=\phi^*_2(\x_3)-\phi^*_3(\x_2)\end{equation}
which implies \[[G_1]-[G_2 ]=[G_U \otimes_k k[\e]]-[G_U \otimes_f k[\e]]\] as an element in $(t_G\otimes t_{G^*}) (k)$. 

Therefore, the difference of the two classes $[G\otimes_{\p_1}A'_{|x'}]-[G\otimes_{\p_2} A'_{|x'}]$ is the same as the difference between the trivial deformation to $k[\e]$ and the deformation given by $s$, which is exactly $\mathrm{ks}(\d_{|x})$.\end{proof}

 Fix a lifting $ C \ra C_2$ of $C$. By \ref{obstruction and torsor}, $\mathrm{ob}_{C_2}(G)=(s_{ij})\in H^1(C, t_G\otimes t_{G^*})$ where $s_{ij}=[G{_{U'_{ij}}}]-[G{_{U'_{ji}}}]$. Then for $\i: \text{Spec }k\ra \text{Spec } W_2,$ by \ref{KS}, the Kodaira-Spencer map induces a bijection 
\begin{align} \label{bijection} \begin{split}
\Def(C, \iota) \ra& H^1(C, t_G \otimes t_{G^*})\\
C'_2\,  \mapsto & \rm{ks}([C'_2]-[C_2]).  \end{split}
\end{align} 
Then there is a lifting of $C$ corresponding to 0 in $H^1(C, t_G\otimes t_{G^*})$. Explicitly, we change the gluing $U'_{ij} \ra U'_i$ by the derivation $\d_{ij}=\rm{ks}^{-1}(-s_{ij})$ and then from \ref{KS}, the new $G_{U'_i}$ is isomorphic to $G_{U'_j}$.  

That is how we adjust the lifting of $C$ to kill the obstruction. It is clear that such a lifting is unique.  Therefore for any BT or BT$_n$ group $G \ra C$, once the curve $C$ is a versal deformation, we obtain a unique first-order deformation of $C$ such that $G$ can be deformed. 

Now we have finished the first order thickening.

\section{Higher order deformation} \label{higher order deformation}

\subsection{Step III} \label{higher order}
For the higher order deformation problem, we would like to complete the following diagram
\[\xymatrix{
G\ar[r] \ar[d]&G_n \ar[d] \ar@{-->}[r] & ? \ar@{-->}[d]\\
C \ar[r]& C_n\ar[r] & C_{n+1}.
}\] It is easy to check that \ref{obstruction and torsor}(1) is still true for the higher order deformation, and thus the obstruction class $\mathrm{ob}_{C_{n+1}}(G_n)$ is still in $H^1(C, T_C)$. 

\begin{prop}\label{higher KS} Theorem \ref{KS} is also true for higher order thickening. 
\end{prop}
\begin{proof}We prove it by induction. 

Based on the induction hypothesis that the proposition is true for $n-$th order thickening, we just mimic the proof of \ref{KS}. Similar to the diagram \ref{the geometric diagram}, we can adjust the lifting $C_n \ra C_{n+1}$ by 

\[\xymatrix{
U_{n+1} & (U_{n+1}\cap V_{n+1}) \ar@<2pt>[l] \ar@<-2pt>[l] \ar[r] & V_{n+1}\\   
U_n \ar[u] & U_n\cap V_n \ar[u] \ar[l] \ar[r] & V_n \ar[u] 
}\]

In terms of algebra, correspondingly, the diagram (\ref{the diagram}) becomes : 
\[\xymatrix{
B\ar@{-->}[r]^\d \ar[d]^{.p^n}& A \ar[d]^{. p^n} \ar[r]^{/\frak{m}_x} & k \ar[d]^{.p^n}\\
B_{n+1} \ar@<2pt>[r]^{\p+p^n \d} \ar@<-2pt>[r]_\p \ar[d] & A_{n+1} \ar[r]^{/I_{x_{n+1}}} \ar[d]& W_{n+1}(k) \ar[d]\\
B_n \ar[r] & A_n \ar[r]^{/I_{x_n}} & W_n
}\]where $\d_{|x}$ is given by $f: k[\e] \ra U\cap V, \e \mapsto s \in \frak{m}_x/\frak{m}^2_x$. 

Let $G_{U_{n+1}}$ be any deformation of $({G_n})_{U_n}$ to $U_{n+1}$. Let
\[I_1:=\psi^{-1}(I_{x_{n+1}}), I_2:=(\psi+p\d)^{-1}(I_{x_{n+1}}).\]
Then $B_{n+1}/I_1 \cong B_{n+1}/I_2 \cong W_{n+1}(k).$We use $\tilde s$ to denote the lifting of $s$ to $B_{n+1}$ such that $\tilde s\in I_1$.
Then 
\[I_2=(\tilde s-p^n), I_1 =(\tilde s).\]

Let 
\[G_1:= G_{U_{n+1}} \otimes B_{n+1}/I_1, G_2:= G_{U_{n+1}} \otimes B_{n+1}/I_2.\]

Choose $I=(\tilde s^2, p\tilde s)$ and $J=(\tilde s, p^n)$. Similar to diagram \ref{the local deformation}, we have 

\[\xymatrix{
&&B_{n+1}/I_1 \ar[dl]&\\
k &W_n \ar[l]&B_{n+1}/I_2\ar[l]& B_{n+1}/I \ar[dl]^{\phi_3} \ar[ul]_{\phi_1} \ar[l]_{\phi_2}\\
 &&W_n[\e]/(p\e) \ar[ul]
}\] where $\phi_1(\tilde s)=0, \phi_2(\tilde s)=p^n, \phi_3(p^n)=0$ and $(B_{n+1}/I)/J\cong W_n$. 

As in the proof of \ref{KS}, let 
\[\xi_1=[G_{U_{n+1}}\otimes B_{n+1}/I]-[G_1 \otimes B_{n+1}/I]\]
\[\xi_2=[G_{U_{n+1}} \otimes B_{n+1}/I]-[G_2 \otimes B_{n+1}/I]\]
\[\xi_3=[G_2\otimes B_{n+1}/I]-[G_1 \otimes B_{n+1}/I]\]
where $\xi_i \in t_{G_x}\otimes t_{G^*_x} \otimes J \cong k^2$.

Similar to Equation \ref{key equation}, we have 
\[\phi^*_2 =\phi^*_1+\phi^*_3.\]
Hence the same result as Equation \ref{the final equation} holds for the higher order deformation.

We obtain that 
\[[G_{U_{n+1}}\otimes B_{n+1}/I_1]-[G_{U_{n+1}}\otimes B_{n+1}/I_2]= [G_{n}\otimes B_n/I]-[G_n\otimes_{B_n/{I_{x_n}}} W_n(\e)].\] The right hand side is the deformation of $G_{|U_n}\otimes W_n$ and  mod $p$, it is just 
\[[G\otimes_f k[\e]]-[G\otimes_k k[\e]]=\rm{ks}(\d).\] 
Therefore by functorality, the right hand side 
\[ [G_n\otimes B_n/I]-[G_n\otimes_{W_n} W_n(\e)]=\rm{ks}(\d). \]

Hence the result of \ref{KS} is also true for the higher order deformation. 
\end{proof}

Let $\iota: \text{Spec }W_n \ra \text{Spec }W_{n+1}$. Proposition \ref{higher KS} implies (\ref{bijection}) holds in the higher order case: 
\[\Def(C_n, \iota) \ra H^1(C,t_G\otimes t_{G^*})\] 
is bijective. Then we can always adjust the lifting $C_{n+1}$ so that it admits a lifting $G_{n+1}$ as BT or BT$_n$ groups. 

In the following, we only consider the BT group case. By \ref{the uniqueness remark}, for a BT group $G$, the lifting of $G$ to $C_2$ is strongly unique. Since \ref{reduce to truncated 0}(2) and \ref{torsor of BT} also hold for higher deformation case, the lifting of BT group $G$ is strongly unique for the deformation problem $C \hookrightarrow C_{n-1} \hookrightarrow C_n$. In this case, we have a canonical choice of deformations $\{C_n\}$ of the curve $C$ such that each $C_n$ admits the unique BT group $G_n$. 
\begin{prop} \label{uniqueness for BT}
For a BT group $G$, the sequence of deformations $\{C_n\}$ of curves $C$ which admits a deformation of $G$ is unique. 
\end{prop}

\subsection{Step IV}
In this step, we always assume  $G$ is a BT group. By \ref{uniqueness for BT}, we have a formal scheme \[\hat G=\lim G_n \ra \hat C=\lim C_n \ra \rm{Spf}(W).\] A natural question is whether there exists actual schemes \[G'\ra C' \ra W\] that base change to $W_n$, it is $G_n \ra C_n$. 

Recall the Grothendieck existence theorem: let $A$ be a noetherian ring, $I$ an ideal of $A$ and $Y= \mbox{Spec\,}A$, $Y_n=\mbox{Spec\,}A/I^{n+1}$, $\hat Y=\text{colim}_n Y_n=\mathrm{Spf}(A)$. 

\begin{thm} \label{EGA}(\cite[5.1.4]{EGA} ) Let X be a noetherian scheme, separated and of finite type over $Y$, and let $\hat X$ be its $I-$adic completion.  Then the functor $F \mapsto \hat F$ from the category of coherent sheaves on $X$ whose support is proper over $Y$ to the category of coherent sheaves on $\hat X$ whose support is proper over $\hat Y$ is an equivalence. 
\end{thm}

\begin{thm}(\cite[5.4.5]{EGA} ) Let $\scr{X} \ra \mathrm{Spf}(A)$ be a formal scheme and $\scr{L}$ an ample line bundle on over $\scr{X}$, then $(\scr{X}, \scr{L})$ are algebraizable. 
\end{thm}
 Hence the formal projective curve $\{C_n\}$ is algebraizable. Denote the algebraic curve as $C'$. 

\begin{cor}Let $X/Y$ be as the \ref{EGA}. Then \begin{enumerate}
\item  $Z\ra \hat Z$ is a bijection from the set of closed subschemes of $X$ which are proper over $Y$ to the set of closed formal subschemes of $\hat X$ which are proper over $\hat Y$.
\item $Z \ra \hat Z$ is an equivalence from the category of finite $X-$schemes which are proper over $Y$ to the category of finite $\hat X-$formal schemes which are proper over $\hat Y$.
\item If both of $X$ and $Z$ are proper, then $(f:X\ra Z) \mapsto (\hat f: \hat X \ra \hat Z)$ is a bijection from the set of morphisms between $X$ and $Z$ to the set of formal morphisms between $\hat X $ and $\hat Z$. In particular, any finite locally free group schemes over $X_n$ can be algebraized to a finite locally free group scheme over $X$. 
\end{enumerate}
\end{cor}

\begin{proof} The first and second assertions are  (\cite[ 4.5, 4.6]{Ill 2}). For any $f_n: X_n \ra Z_n$ morphism between formal schemes, the graph of morphism: $(X_n \cong\,) \G_n\subset X_n \times Z_n$ can be algebraized to $\G\subset X\times Z$ which is still isomorphic to $X$ and then the morphism can also be uniquely algebraized. \end{proof}

\begin{prop}
The formal scheme \[\hat G \ra \hat C \ra \mathrm{Spf}(W)\] can be algebraized.
\end{prop}
\begin{proof}
First consider the truncated case, we have an algebraization $G' \ra C'$ of $G_n \ra C_n$ where $G'$ is a finite locally free group scheme over $C'$. It remains to show $G'$ is BT$_n$. It suffices to show it satisfies the exact sequence for any $l\leq k$: 
\[0\ra G'[p^l] \ra G' \ra G'[p^{k-l}] \ra 0.\] Since the algebraization of closed subschemes and morphisms are unique up to isomorphism, the algebraization of $G_n[p^l] \ra G_n$ gives a $p^l-$torsion subgroup $G'[p^l]\ra G'$ and the composition $G'[l] \ra G' \ra G'[p^{k-l}]$ is trivial. Furthermore, the morphism $G_n/G_n[p^l]\cong G_n[p^{k-l}]$ can also be algebraized.   

For the BT groups, any BT group $G_n$ can be represented as a colimit $\text{colim}_k G_n[p^k]$ and $G_n[p^k]$ can be algebraized to $G'[p^k]$ for each $k$. Since the algebraization preserves the morphism, $G'[k]$ is still an inductive system and $\text{colim}_k G'[p^k]$ is a BT group. Hence the BT group $G_n$ can be algebraized too. \end{proof}

\subsection{The truncated case} \label{the truncated case}
For a BT$_n$ group $G$ versally deformed over curve $C/k$, by  \ref{higher KS}, $G$ is unobstructed. But since \ref{obstruction and torsor}(2) may not be true for BT$_n$,  the lifting of $G$ may not be unique and consequently the choice of $C_m$ may not be unique if $m$ is large enough. But we still have the following weaker result: 

\begin{thm}
If $C$ is a versal deformation of a BT$_n$ group $G$, then there exists unique curve $C_n/W_n$ and BT$_n$(may not be unique) $G_n/C_n$ such that the following diagram
\[\xymatrix{
G \ar[d] \ar[r] & G_n \ar[d]\\
C \ar[r] \ar[d]& C_n \ar[d]\\
k \ar[r] & W_n
}\] 
is a fiber product. 
\end{thm}

\begin{proof}
We have known the existence from Step I, II and III. So it suffices to show the uniqueness of $C_n$. We prove it by induction. 

From \ref{KS}, $C_2$ is unique for any $n\geq 1$. By (\cite[ Theorem 4.4(d)]{Ill}), we know for any $G' \in\Def_{C_2}(G)$, the morphism 
\begin{equation}\label{trivial auto} \mathrm{Aut}(G') \ra \mathrm{Aut}(G'[p^{n-1}])\end{equation} 
is trivial. Note we have the morphism \ref{surj between torsors} is null and $H^0(C, t_G \otimes t_{G^*})\cong H^0(C, T_C)=0$. Hence for any $G', G'' \in\Def_{C_2}(G)$, $G'$ and $G''$ are locally isomorphic. Restricting  to $G'[p^{n-1}]$ and $G''[p^{n-1}]$, these local isomorphism are unique and then compatible on the overlapping. Hence we have 
\[G'[p^{n-1}] \cong G''[p^{n-1}]\] globally. By \ref{reduce to truncated}, all the deformation to $C_2$ has the same obstruction class to the second order deformation. Then by \ref{higher KS}, there exits a unique $C_3$ admits the deformation.

Assume that $C_{k-1}$ is unique with $k-2<n$ and all the deformations in $\Def_{C_{k-1}}(G)$ share the same $p^{n-k+2}-$kernel, i.e. for any $G', G'' \in\Def_{C_{k-1}}(G)$, globally
\[G'[p^{n-k+2}] \cong G''[p^{n-k+2}].\]
Then by \ref{reduce to truncated} and \ref{higher KS}, the obstruction to deform $G'[p^{n-k+2}]$ from $C_{k-1}$  gives a unique lifting $C_{k-1} \ra C_k$. And by the same argument in the $C_2$ case, the  different lifting of $G'[p^{n-k+2}]$ share the same $p-$ kernel, i.e. for any $K' ,K'' \in \Def_{C_{k}}(G)$, \[K'[p^{n-k+1}] \cong K'[p^{n-k+1}]. \] So we can keep induction until $k=n$.
\end{proof}

For higher order deformation, there may not exist a canonical choice of $C_m$. However, since the lifting of BT$_n$ is always unobstructed, using the result in Step IV,  we have 

\begin{thm}
For a BT$_n$ group $G$ versally deformed over $C$, there exists a curve $C'$ over $W(k)$ and a BT$_n$ group $G' \ra C'$ such that the following diagrams  
\[\xymatrix{
G \ar[r] \ar[d] & G' \ar[d]\\
C\ar[r] \ar[d] &C' \ar[d]\\
k \ar[r]& W
} \] are both fiber products.
\end{thm}

\section{Equi-characteristic case} \label{equicharacteristic}
We have an analogue to \ref{main theorem} in the equi-characteristic case.

\begin{thm} \label{rigidity}
Assumption as \ref{main theorem}, for any $n\geq 2$, the unique deformation (up to isomorphisms) of $G\ra C$ to $k[\e]/(\e^n)$ is the trivial deformation $G\times_k k[\e]/\e^n \ra C\times_k k[\e]/\e^n$, i.e. $G\ra C$ is rigid. 
\end{thm}
The proof of \ref{rigidity} largely parallels the mixed-characteristic case. We consider the equi-characteristic first-order deformation $$\text{Spec }k \stackrel{i}{\ra} \text{Spec }k[\e]/\e^2.$$ The fiber product
\[\xymatrix{
G_{2,0} \ar[r] \ar[d]& G \ar[d]\\
C_{2,0} \ar[r] \ar[d] & C \ar[d]\\
\text{Spec } k[\e]/\e^2 \ar[r] & \text{Spec }k
}\] gives the trivial deformation of the BT group since the fiber product preserves the BT structure.

Since we can carry on the proof word by word except replacing $W_2$ by $k[\e]/\e^2$, Theorem \ref{obstruction and torsor} is  true for the equi-characteristic case. And further, we can show \ref{KS} is also true in the equi-characteristic case. 

Assume $C=U\cup V$ as union of open affine schemes and $U'(\text{resp.} V')$ are the unique deformation of $U(\text{resp.}V)$ to $k[\e]/\e^2$ as in the proof of \ref{KS}. Comparing the deformation of $C$ with the trivial deformation: locally,
\[\xymatrix{
B':=B[\e]/\e^2 \ar@<2pt>[r]^{l+\e \d} \ar@<-2pt>[r]_{l} \ar[d] & A[\e]/\e^2 \ar[d] \ar[r]^{\pi} & k[\e]/\e^2 \ar[d] \\
B \ar[r]^l & A \ar[r]^{/\frak{m}_x} & k\\
}\]
where similar to \ref{the i}, $\d$ induces $f:B \ra k[\e]/\e^2 $. Then 
\[ \pi\circ (l+\e\d) =f.\]
Hence \[G_{U'} \otimes_{\pi\circ (l+\e\d)}k[\e]/\e^2=G_{U'} \otimes_f k[\e]/\e^2. \]That justifies that 
\[\rm{ks}(\d)=[G_{U'} \otimes_{l+\e\d} A[\e]/\e^2]-[G_{U'} \otimes_l A[\e]/\e^2].\]

Therefore, the Kodaira Spencer map induces an isomorphism:
\begin{align*}
\Def(C, i)\ra& H^0(C, t_G \otimes t_{G^*})\\
C_2 \mapsto& {ks}([C_2]-[C_{2,0}])
\end{align*}Hence only the trivial deformation of $C$ admits a lifting of $G$. Theorem \ref{obstruction and torsor}(2) shows such a lifting is unique. 

By induction, we can prove it for the higher order deformation $\text{Spec }k \ra \text{Spec }k[\e]/\e^n$. For any deformation $G_n \ra C_n \ra k[\e]/\e^n$, base change to $k[\e]/\e^{n-1}$, it is the trivial deformation on $\text{Spec }k[\e]/\e^{n-1}$. Then we can carry the argument in the Section \ref{higher order} word by word to the current situation. Therefore \ref{rigidity} is proved. 

\section{A height 2 BT group with trivial Kodaira Spencer map}
\label{trivial KS map}

Another extreme case is $G$ with trivial Kodaira-Spencer map, or equivalently, a rank 2 Dieudonne crystal over $C$ with trivial Higgs field. Let $\cV$ be the Dieudonne crystal $\DD(G)$ associated to $G$ with the Hodge filtration 
\[0\ra \sL \ra (\cV)_C \ra \sL' \ra 0 \]
and trivial Higgs field $\theta: \sL \ra \sL'\otimes \O^1_C$. Let $F$ and $V$ be the Frobenius and Verschiebung for $\cV$.  Then over $C$, $\ker F=\im V=\sL^p $. 

\begin{thm}
If a height 2 BT group $G$ has trivial Kodaira-Spencer map, then there exists a BT group $H$ such that $H^{(p)}=G$.  
\end{thm}

\begin{proof}
Since $\theta=0$, the connection preserves the line bundle $\sL$.  Choose a lifting $ C'$ of $C$ to $W$. Then the crystal $\cV$ corresponds to a bundle with an integral connection. For notational simplicity, we still denote it as $(\cV, \nabla)$. Let $L=\pi^{-1}(\sL)\subset \cV$ be the inverse image of $\sL$ under $\pi: \cV \ra \cV_C$. Since $\nabla$ preserves $\sL$, $\nabla$ preserves $L$. 

For any affine open subset $U\subset  C'$, choose a lifting of the absolute Frobenius $\s$ to $U$. Then $L^\s \subset \cV^\s$. Since $\pi(L^\s)=\pi(V(\cV))$ where $V$ is the Verschiebung, we have $L^\s = V(\cV) \cong \cV$. In particular, $L^\s$ is isomorphic to a rank 2 vector bundle over $U$. It implies that $L$ can be viewed as a vector bundle over $ C'$. 

It remains to define the Frobenius $F_L$ and Verschiebung $V_L$ for $L$ locally. Over any $U$, $V: \cV^\s \ra \cV$ induces just the inclusion $L \ra L^\s\cong \cV$, which can be chosen to be $V_L$. And $F(L^\s)=F(V(\cV))=p.\cV \subset L$. Hence $F$ induces a morphism $F_L: L^\s \ra L$. Since both of $F_L$ and $V_L$ are restriction of $F$ and $V$, we have $F_L \circ V_L = V_L \circ F_L = p. \text{Id}$. 

Therefore $(L, \nabla, F_L, V_L)$ corresponds to a rank 2 Dieudonne crystal whose Frobenius pullback is isomorphic to $\cV$. By (\cite[ Main Theorem 1]{deJ}),  such a crystal corresponds to a height 2 BT group $H$ and $H^{(p)}=G$.
\end{proof}

\section{Application: lifting to a Shimura curve  } \label{application}
\subsection{Shimura curves of Mumford type}
In \cite{Mum}, Mumford defines a family of Shimura curves. We briefly recall the construction as follows. Let $F$ be a cubic totally real field and $D$ be a quaternion division algebra over $F$ such that 
\[D\otimes_\Q \R \cong \mathbb{H}\times \mathbb{H} \times M_2(\R), \rm{Cor}_{F/\Q}(D)\cong M_8(\Q).\]
Here $\mathbb{H}$ is the quaternion algebra over $\R$. 

Let $G=\{x \in D^* | x.\bar x=1\}$. Then $G$ is a $\Q-$simple algebraic group. And it is the $\Q-$form of the $\R-$algebraic group $SU(2)\times SU(2)\times SL(2,\R)$. 
\begin{align*}
h: \mathbb{S}_{m}(\R) \ra& G(\R)\\ e^{i\th} \mapsto&I_4 \otimes \begin{pmatrix} \cos \th & \sin \th \\-\sin \th & \cos \th \\ \end{pmatrix}.
\end{align*} 

The pair $(G, h)$ forms a Shimura datum. And it defines a Shimura curve, parameterizing abelian fourfolds. Generalizing the construction, one is able to define some Shimura curves of Hodge type, parameterizing $2^m$ dimensional polarized abelian varieties (see \cite{Zuo}). We call such curves (with its universal family) \textit{the Shimura curves of Mumford type}.

\begin{defn}
Define the binary operation between two BT groups: \[G\odot H:=\colim_n (G[p^n]\otimes_\Z H[p^n]).\]
\end{defn}

\begin{rmk}
The inductive system $(G[p^n]\otimes_\Z H[p^n])$ will be explained in the Appendix \ref{the dot product}. Note in general $G\odot H$ is just an abelian sheaf rather than a group scheme. But in our case, $H$ is etale and we will show in the Appendix \ref{the dot product} that $G\odot H$ is indeed a BT group and $(G\odot H)[p^n]=G[p^n]\otimes_\Z H[p^n]$.
\end{rmk}

\begin{thm} \label{lifting in general pol}
Let $X \ra C$ be a family of abelian varieties over a smooth proper curve $C$ such that
\begin{enumerate}
\item $X/C$ has the maximal Higgs field, 
\item there exists an isomorphism 
\[X[p^\infty] \cong G\odot H\] where $G$ is a height 2 BT group and $H$ is an etale BT group over $C$.
\end{enumerate}
Then $X/C$ is a good reduction of a Shimura curve of Mumford type. 
\end{thm}

We devote the rest of the section to prove this theorem.  We first apply \ref{main theorem} to lift $X\ra C$ to a formal family over $W$. Secondly, to algebraize the family, we show any polarization lifts as well.

\begin{rmk}
There exist some good reductions of Shimura curves of Mumford type satisfying the conditions of $\ref{lifting in general pol}$. We will show it in our next paper. 
\end{rmk}
\subsection{Lifting of $X/C$}

 Let $\cT$ be the unit crystal $\DD(H)$ and $\cV$ be the Dieudonne crystal $\DD(G)$. 
\begin{lemma} \label{confused part}
 We have $\DD(G\odot H)=\cV \otimes \cT$ and the Hodge filtration of $\DD(X[p])$ is just 
 \[0\ra \o_{G}\otimes \cT \ra \DD(X[p])_C \ra t_{G^*}\otimes \cT \ra 0\]
\end{lemma}

\begin{proof} Since $H$ is etale, for any $n$,  there exists a finite etale morphism $f_n: T_n \ra C$ such that $f^*_n (H[p^n])$ is trivial.

Over $T_n$, 
\[\tag{*}\DD_{T_n}(f^*_n(\bGn))\cong f^*_n(\cV/p^n)^{\oplus \Ht(H)}\cong f^*_n \cV/p^n \otimes_{\cO_C} f^*_n \cT/p^n\] as Dieudonne crystals.
Both sides are effective descent datum with respect to $T_n \ra C$. For any $g\in \mathrm{Aut}(T_n/C)$, $g^*$ acts on both of $f^*_n(\bGn)$ and $f^*_n \cV/p^n \otimes_{\cO_C} f^*_n \cT/p^n$ which is compatible with the functor $\DD_{T_n}$: 
\[\xymatrix{
f^*(\bGn)\ar[r]^{g^*} \ar[d]^{\DD_{C'}}& f^*(\bGn) \ar[d]^{\DD_{C'}} \\
f^*(\cV \otimes \cT/p^n) \ar[r]^{g^*} & f^*(\cV \otimes \cT/p^n)
}
\] is commutative. Hence the isomorphism (*) between effective descent datum also descends to $C$. 

Then we have for any $n,$
\[\DD(\bGn)=(\cV \otimes \cT/p^n, F_\cV \otimes F_\cT, V_\cV \otimes F^{-1}_\cT).\] 

 So $\DD(G\odot H)=\cV \otimes \cT$ and in particular, we have a tensor decomposition of Frobenius: 
\[F_{\DD(X[p])}= F_{\DD(G)} \otimes F_\cT.\] Note $\ker F_{\DD(G)}=\o_{G[p]}$ and $F_\cT$ is isomorphic. Hence $\ker (F_{\DD(G)} \otimes F_\cT)=(\o_{G[p]}\otimes \cT)^{(p)}. $ By (\cite[2.5.2]{deJ}), the Hodge filtration is just given by $\o_{G}\otimes \cT \ra \DD(X[p])_C$. 
\end{proof}

\begin{lemma} \label{verdeform}
The curve $C$ is a versal deformation of $G$.
\end{lemma}

\begin{proof}
Consider the Hodge filtration \[0\ra \o_{X[p]} \xa{i} \DD(X[p])_C \ra t_{X[p]} \ra 0.\] 
By \ref{confused part}, we have $\o_{X[p]}\cong \o_{G} \otimes \cT$ and $i$ induces an isomorphism on $\cT$. Hence the Higgs field $\theta_{\DD(X[p])}$ on $X[p]$ actually comes from the Higgs field $\theta_{\DD(G)}$ on $G$:
\[\theta_{\DD(X[p])} \cong \theta_{\DD(G)} \otimes \mathrm{id}.\]
Hence the maximal Higgs field descends to $G/C$.  By (\cite[A 2.3.6]{Ill}), $C$ is a versal deformation of $G$. 
\end{proof}

\begin{rmk} \label{fix the lifting}
By \ref{main theorem}, $C$ has a unique lifting $C'$ to $W$ such that $C'$ admits a lifting of $G$. We fix this lifting once for all and take 
\[C \ra C_n \] as the corresponding lifting to $W_n$ .
\end{rmk}

\begin{prop} \label{formal family}
The family $X \ra C$ can be lifted to a formal abelian scheme $\{X_n \ra C_n\}$ over $W$.
\end{prop}

\begin{proof}
Since $H$ is etale, it automatically lifts to $W(k)$. Combining with \ref{main theorem}, $X[p^\infty]$ can be uniquely lifted to $C_n$ .  By (\cite{Messing}),we have a formal abelian scheme $\{X_n \ra C_n\}$. 
\end{proof}

\subsection{Lifting of polarization} \label{lifting of polarization}
Choose any polarization $\l: X \ra X^t$ over $C$ and fix it throughout the rest of the paper.  Once we can lift the polarization $\l$, the formal scheme in \ref{formal family} can be algebraized. Since $\l$ is an isogeny,  it suffices to lift the sub group scheme $\ker \l$. We can always assume $\l_1: X[p] \ra X^t[p]$ induced by $\l$ is not zero. Otherwise, $X[p] \subset \ker \l$ and since $X[p]$ always lifts, we can replace $\l$ by $\frac{1}{p} \l$. 
 
By \ref{verdeform}, the curve $C$ is a versal deformation of $G$. For any closed point $c$, $\hat C_c$ is the versal deformation space $\Def(G_c)$. Theorem 3.2 in \cite{Oort} indicates that the supersingular loci in $\Def(G_c)$ has only finite points. Therefore $G$ is generically ordinary over $C$.  

Since $X[p^\infty]\cong G \odot H $, 
\[X^t[p^\infty]\cong G^t \odot H^t\]
 where $H^t$ corresponds to the dual Galois representation of $H$. 

Let $\eta$ be the generic point of $C$ and $h$ be the height of $H$. Let $\cV$ be the Dieudonne crystal $\DD(G)$ and $\cT$ be the unit root crystal $\DD(H)$. Then $\cV_C$ has Hodge filtration
\[0\ra \sL_1 \ra \cV_C \ra \sL_2 \ra 0.\]

\begin{lemma} \label{d=0}
If a morphism $d: \cV_C\ra \cV^\vee_C$ preserves Hodge filtration, conjugate filtration and Higgs field while induces a zero morphism in $\Hom(\sL_2, \sL^\vee_1)$, then $d=0$.
\end{lemma}

\begin{proof}
 Consider the following diagram
\[\xymatrix{
\cV^\s_C \ar[r]^F \ar[d]^{d^\s} & \cV_C \ar[d]^d\\
\cV^{\vee,\s}_C \ar[r]^{F^\vee} & \cV^\vee_C .\\
}\] Note this diagram is commutative since $d$ commutes with $F$.

On one hand, since $d$ preserves Hodge filtration and induces zero on cokernels, $d$ factors through $\sL_2 \ra \sL^\vee_2$. And $F^\vee: \cV^{\vee, \s}_C \ra \cV^\vee_C$ factors through $\sL^{\vee, \s}_1$.  Therefore $F^\vee \circ d^\s=0.$ 

On the other hand, $F$ induces $\sL^\s_2 \ra \cV_C$. Since $\cV_C$ has maximal Higgs field, we have 
\[\xymatrix{
&\sL^\s_2 \ar[d] \ar[dr]&\\
\sL_1 \ar[r]&\cV_C \ar[d] \ar[r] & \sL_2\\
&\sL^\s_1& 
}\] the induced map $\sL^\s_2 \ra \sL_2$ is generically surjective. In order for $d\circ F=0$, we must have $d=0$ since $d$ factors through $\sL_2$. 
\end{proof}

\begin{cor} \label{k}
Over any finite etale covering $T$ over $C$,
\[\End(\cV_T, \sL_{1\,T}\hookrightarrow \cV_T, \sL^\s_{2\,T} \hookrightarrow \cV_T, \theta_T)=k.\]
\end{cor}

\begin{proof}
For any $\phi \in \End(\cV_T, \sL_{1\,T}\hookrightarrow \cV_T, \sL^\s_{2\,T} \hookrightarrow \cV_T, \theta_T)$, $\phi$ induces a map $s\in \End(\sL_{1\,T})=k$. Since $\phi$ preserves $\theta$ and $\theta: \sL_1 \ra \sL_2 \otimes \O^1_C$ is isomorphic, the map induced by $\phi$ in $\End(\sL_2)$ is also $s$.

Consider $\phi-s.$Id. It preserves Hodge filtration, conjugate filtration and Higgs field while induces a zero morphism in $\End(\sL_1)$. Since $T\ra C$ is etale, the pull back Higgs field $\theta_T$ is still maximal. Therefore the conclusion in \ref{d=0} is still true if we consider $\cV_T$ instead of $\cV_C$. And hence $\phi=s. \text{Id}$. 
\end{proof}
\begin{lemma} \label{nonsplitting}
The sequence associated to the ordinary group scheme 
\[0\ra G^\mult[p]_{\eta} \ra G[p]_\eta \ra G^\et[p]_{\eta} \ra 0\]
does not split.
\end{lemma}

\begin{proof}
If the sequence splits, then $\DD(G[p]) \cong \DD(G^\mult[p]_{\eta}) \times \DD(G^\et[p]_{\eta})$. So the Hodge filtration of $\DD(G[p]_\eta)$ is the direct sum of that of $\DD(G^\et[p]_{\eta})$ and $\DD(G^\mult[p]_{\eta})$. Therefore  
\[0 \ra \DD(G^\et[p]_{\eta}) \ra \DD(G[p]_\eta) \ra \DD(G^\mult[p]_{\eta})\ra 0\] is the Hodge filtration. Since Gauss-Manin connection preserves this filtration, it contradicts to the maximal Higgs field. 
\end{proof}

\begin{lemma} \label{subgroup}
For any subgroup scheme $K \subset G[p]_\eta$, $K $ can only be $0, G[p]_\eta$ or $G^{\mult}[p]_{\eta}$.
\end{lemma}

\begin{proof}
Since $G[p]_\eta$ has height 2, any nonzero proper subgroup scheme $K$ can only have height 1 and order $p$. Note both of $G^\mult[p]_{ \eta}$ and $G^\et[p]_{\eta}$ have order $p$ and height 1.

Consider the diagram

\[\xymatrix{
&K \ar@{^{(}->}[d]^i \ar[dr]& \\
{G^\mult[p]_{ \eta}} \ar[r] &{G[p]_\eta} \ar[r]^F & {G^\et[p]_{ \eta}}.
}\]
If $K \cap G^\mult[p]_{\eta}$ is not empty, then $G^\mult[p]_{ \eta} \subset K$. Since $K$ has height 1, $G^\mult[p]_{\eta} = K$.

If $K \cap G^\mult[p]_{\eta}$ is empty, then $F \circ i$ is an injection. Comparing the height, it is further an isomorphism. But it contradicts to \ref{nonsplitting}. 
\end{proof}

Again as in the proof of \ref{k}, for any finite etale covering $T \ra C$, the pull back Higgs field $\theta_T$ is still maximal. Therefore we can carry the proof of \ref{nonsplitting} and \ref{subgroup} on $T$. And we have the following corollary. 
\begin{cor}\label{subgroup over etale}
For any finite etale covering $T$ of $C$, we use $G_T$ to denote the pull back of $G$ to $T$ and hence for any subgroup scheme $K \subset G_T[p]_\eta$, $K $ can only be $0, G_T[p]_\eta$ or $G^{\mult}_T[p]_{\eta}$.
\end{cor}
\begin{lemma} \label{the pair}
There exists a pair $\{\tilde C, \g\}$ such that 
\begin{enumerate}
\item $\tilde C \ra C$ is a finite etale covering ;
\item Let $\tilde \cV$(resp. $\tilde G$) denote the pull back Dieudonne crystal (resp. BT group) of $\cV$(resp. $G$) from $C$ to $\tilde C$. And $\g: \tilde \cV^\vee \ra \tilde \cV $ is an isomorphism. 
\end{enumerate}
\end{lemma}

\begin{proof}
Firstly we consider $\wedge^2 \cV$. It is a $F-$crystal over $C$ and hence corresponds to a triple $(\sL, \nabla, F)$, an invertible sheaf over $C'$ with integrable connection and Frobenius.  Obviously as $F-$isocrystal, $\sL$ is isoclinical of slope $p$. Replace $F$ by $p^{-1}F$ and then $\sL$ is unit root. Thereby it is equivalent to a one dimensional monodromy 
\[\pi_1(C, \bar c) \ra \Z^*_p.\] Note $\pi^{\rm{ab}}_1(C ,\bar c)$ is an extension of $\hat{\Z} $ with a finite group. So the image of the monodromy is just a finite subgroup of $\Z^*_p$.  And then there exists a finite etale covering $\tilde C \ra C$ such that the pull back of the monodromy to $\tilde C$ is trivial. 

Let $\tilde \cV$ denote the pull back of $\cV$ to $\tilde C$. Then $\wedge^2 \tilde\cV$ is trivial as an $F-$crystal which implies as a $F-$isocrystal, $\tilde \cV$ is self-dual. Consider the following morphism 
\[\tilde \cV \ra \tilde \cV \otimes \tilde \cV \otimes \tilde \cV^\vee \ra \wedge^2 \tilde \cV \otimes \tilde \cV^\vee \cong \tilde \cV^\vee.\] Obviously it is nonzero and induces a morphism between BT groups $\tilde G^t \ra \tilde G$. By $\ref{subgroup over etale}$, the morphism $\tilde G^t \ra  \tilde G$ is isomorphic. So we have found an isomorphism $\g: \tilde \cV^\vee \ra \tilde \cV.$
\end{proof}

Let $\cT^\vee$ be the unit root crystal $\DD(H^\vee)$ corresponding to the dual $H^\vee$ and $\tilde H$ (resp. $\tilde \cT$) be the pull back of $H$ (resp. $\cT$) to $\tilde C$. Since $\tilde H$ is an etale BT group, for each $n$ there exists a finite etale covering $T_n$ of $\tilde C$ such that  the pull back of $\tilde H[p^n]$ to $T_n$ is trivial. Then 

\[T_n \ra T_{n-1} \ra \cdots \ra T_1 \ra \tilde C\] is a tower of finite etale coverings of $\tilde C$.

Then the polarization $\l_1: X[p] \ra X^t[p]$ induces $\tilde \l^*_1: \tilde \cV^\vee_{\tilde C }\otimes \tilde \cT^\vee_{\tilde C} \ra \tilde \cV_{\tilde C }\otimes \tilde \cT_{\tilde C}$. 
Let $\g_1$ be the restriction of $\g$ to $\tilde \cV^\vee_{\tilde C}$.

\begin{lemma} \label{decomposition of l}
The morphism $ \tilde \l^*_1$ has a tensor decomposition
\[\tilde \l^*_1 = \g_1 \otimes \b_1\] where  $ \b_1$ is a morphism between $\tilde\cT^\vee_{\tilde C} $ and $\tilde\cT_{\tilde C}$.  
\end{lemma}

\begin{proof} Firstly we consider the analogous result over ${T_1}$. Let $\l_{T_1}$(resp. $\cV_{T_1}$, $\cT_{T_1}$) be the base change of $\l_1$(resp. $\cV_C$, $\cT_C$) to ${T_1}$. Then $\l^{*}_ {1\,{T_1}}$ induces a morphism between Hodge filtrations:
\[\xymatrix{
0\ar[r]& \sL_{1\,{T_1}} \otimes \cT_{T_1} \ar[r] & \tilde \cV_{T_1} \otimes \cT_{T_1} \ar[r] &  \sL_{2\, {T_1}}\otimes \cT_{T_1}  \ar[r] & 0\\
0 \ar[r]& \sL^\vee_{2\, {T_1}} \otimes \cT^\vee _{T_1} \ar[r] \ar[u]& \tilde \cV^\vee_{T_1} \otimes \cT^\vee_{T_1} \ar[r] \ar[u]^{\l^*_{1\,{T_1}}} & \sL^\vee_{1\, {T_1}} \otimes \cT^\vee_{T_1} \ar[r] \ar[u]& 0 .
}\] 
On one hand, since $\cT_{T_1}$ is trivial and $\l^*_{1\, T_1}$ is nonzero, by \ref{k}, $\l^*_{1\, {T_1}}$ induces an isomorphism $\g' \in \Hom(\cV^\vee_{T_1}, \cV_{T_1})$, preserving the Hodge filtration, the Frobenius and the Higgs field. 

On the other hand,  we have that $\sL_{1\, {T_1}} \cong \sL^\vee_{2\,{T_1}}$ by \ref{the pair}, and hence  $\Hom(\sL_{1\,{T_1}} \otimes \cT_{T_1}, \sL^\vee_{2\,{T_1}} \otimes \cT^\vee_{T_1})\cong \Hom(\cT_{T_1}, \cT^\vee_{T_1})$. Let $\b'\in \Hom(\cT_{T_1}, \cT^\vee_{T_1})$ be the corresponding morphism. Thus there exists an isomorphism $t\in \Hom(\sL_{1\,{T_1}}, \sL^\vee_{2\,{T_1}})$ such that the induced morphism on the kernel(resp. cokernel) is $t\otimes \bar \b'$(resp. $t^\vee \otimes \bar \b'^\vee$). We can adjust $\g'$ by a scalar such that it induces $t$.

Consider the difference $d= \l^*_{T_1}- \g' \otimes \b'$. Then $d$ preserves Hodge filtration, conjugate filtration and the Higgs field while induces zero morphism on the kernel and cokernel of Hodge filtration. By \ref{d=0}, d is zero. So $ \l^*_{T_1}= \g' \otimes \b'$. Again by \ref{k}, we can adjust $\g'$ by a scalar $k$ so that $\l^*_{1\, T_1}\cong \g_{1\,T_1} \otimes \b'$.

Note both of $\l^*_{1\, T_1}, \g_{1\,T_1}$ can descend to $\tilde C$. Thus $\b'$ descends to a morphism $\b_1$ over $\tilde C$ and $\l^*_1 = \g_1 \otimes \b_1$. 
\end{proof}
In the proof of \ref{decomposition of l}, the only special property of the polarization $\tilde\l$ we use is that it preserves the Hodge filtration, Frobenius and the Higgs field. So we obtain the following corollary. 
\begin{cor} \label{general map}
In the \ref{decomposition of l}, we can replace $\tilde \l^*_1$ by any nonzero morphism between $\tilde \cV_{\tilde C}\otimes \tilde \cT$ and $\tilde \cV^\vee_{\tilde C} \otimes\tilde \cT^\vee$ in the category of  $F-$crystals over $\cri(\tilde C/k)$ and the result still holds.
\end{cor}

\begin{prop}
As a morphism between Dieudonne crystals, $\tilde \l^*=\g\otimes \b$ where $\b$ is a morphism between $\cT$ and $\cT^\vee$.
\end{prop}

\begin{proof}
We will prove the tensor decomposition by induction on the group order. Let $ \l^*_n$ be the polarization 
\[ \l^*_n:   \DD(G^t[p^n]) \otimes \DD(H^\vee[p^n])\ra \DD(G[p^n]) \otimes \DD(H[p^n])\]
It is true for $n=1$. If it is true for $n-1$, then for $n$, 
\[\xymatrix{
\DD(\tilde G[p]) \otimes \DD(\tilde H[p])   & \DD(\tilde G^t[p]) \otimes \DD(\tilde H[p])^t \ar[l]^{\g_1 \otimes \b_1}\\
\DD(\tilde G[p^n]) \otimes \DD(\tilde H[p^n])  \ar[u]^{p^{n-1}} & \DD(\tilde G^t[p^n]) \otimes \DD(\tilde H[p^n])^t \ar[l]^{\tilde \l_n} \ar[u]^{p^{n-1}}\\
\DD(\tilde G[p^{n-1}]) \otimes \DD(\tilde H[p^{n-1]}) \ar[u]  & \DD(\tilde G^t[p^{n-1}]) \otimes \DD(\tilde H[p^{n-1}])^t \ar[u]\ar[l]^{\g_{n-1} \otimes \b_{n-1}}.\\
}\]
Base change to $T_n$,  since  $H[p^n]$ is trivial on $T_n$, we can find 
\[\b_{T_n}: \DD(H_{T_n}[p^n])^t \ra \DD(H_{T_n}[p^n])\]
so that the restriction of $\b_{T_n}$ to $\DD(H'_{T_{n-1}}[p^{n-1}])$ yields $\b_{T_{n-1}}$.

Therefore $(\l^*_n)_{T_n}-\g_{T_n} \otimes_{\Z/p^n\Z} \b_{T_n}$ factors through 
\[\a: \DD(G_{T_n}[p])^\vee \otimes \DD(H_{T_n}[p])^t \ra \DD(G_{T_n}[p^{n-1}]) \otimes \DD(H_{T_n}[p^{n-1}]).\]
 
Definitely $\im\a \subset \DD(G_{T_n}[p]) \otimes \DD(H_{T_n}[p])$. If $\a$ is nonzero, by \ref{general map}, $\a=\g_1 \otimes \b'$ where $\b': \DD(H_{T_n}[p])^t \ra \DD(H_{T_n}[p^{n-1}])$. 

Note $\g_1=\g_{T_n}|_ {p^{n-1}\DD(G_{T_n}[p^n])^\vee}$, we have $\a$ can be written as the composition
\[\DD(G^{t}_{T_n}[p^n])\otimes_{\Z/p^n} \DD(H_{T_n}[p^n])^t \xrightarrow{\text{Id}\otimes p^{n-1}} \DD(G^t_{T_n}[p^n]) \otimes_{\Z/p^n}\DD(H_{T_n}[p])^t \xrightarrow{\g_{T_n} \otimes \b'} \DD(G_{T_n}[p^{n-1}])\otimes_{\Z/p^n} \DD(H_{T_n}[p^{n-1}]).\] Thus $ (\l^*_n)_{T_n}-\g_{T_n} \otimes_{\Z/p^n} \b_{T_n}=\g_{T_n} \otimes_{\Z/p^n} p^{n-1} \b'$. So $(\l^*_n)_{T_n}=\g_{T_n}\otimes(\bar\b_{T_n}+p^{n-1}\b')$.

Since both $\l^*_n$ and $\g_{T_n}$ can be descent to $\tilde C$, $\bar\b_{T_n}+p^{n-1}\b'$ can be descent to $\tilde C$ as well. Therefore as a morphism between Dieudonne crystal over $\tilde C$, $\tilde\l^*=\g \otimes \b$.
\end{proof}

By the equivalence between Dieudonne crystals and BT groups over $\tilde C$(\cite{deJ} Main Theorem 1), we know $\tilde \l$ can be decomposed to a tensor product of an isomorphism between $\tilde G$ and $\tilde G^t$ and a morphism between $H$ and $H^\vee$. Apply \ref{main theorem} to $(\tilde C, \tilde G)$ and thus there exists a strongly unique pair $(\tilde C', \tilde G')$ over $W$ which is a lifting of $(\tilde C, \tilde G)$. In particular, the isomorphism between $\tilde G$ and $\tilde G^t$ lifts to an isomorphism $\tilde G'$ and $\tilde G'^t$. Since $H$ is etale over $C$, any morphism between $\tilde H$ and $\tilde H^\vee$ naturally lifts. Therefore the polarization $\tilde \l$ lifts.

For each $n$, let $\tilde \l_n$ denote the lifting of $\l$ to $\tilde X_n/\tilde C_n$. Since  $\Aut(\tilde C_n/C_n)=\Aut(\tilde C/C),$ $\tilde \l_n$ is a descent datum as $\tilde \l$ and hence descends to a polarization $\l_n$ on $X_n/C_n$. 

By \ref{uniqueness}, such a lifting $\l_n$ is unique. Therefore we have a compatible family of polarization $\{\l_n\}$ and hence the formal family $\{X_n \ra C_n\}$ can be algebraized to a family of abelian varieties $X'\ra C'$ over $W$. 

\begin{rmk}
If we know in priori that $X/C$ is a family of principally polarized abelian varieties, i.e. $\l$ is isomorphic, then $\l$ naturally lifts to any $X_n$ since the lifting $X_n$ of $X$ is strongly unique. 
\end{rmk}

\subsection{Lift to Shimura curves} In the following proposition, we show $X' \ra C'$ is a Shimura curve of Mumford type.

\begin{prop} \label{the Mumford type}
The generic fiber $(X'/C')\otimes \C $ is a Shimura curve of Mumford type.
\end{prop}

\begin{proof}
Obviously the family $X' \ra C'$ has maximal Higgs field.  By Theorem 0.5 in \cite{Zuo}, we directly conclude that $(X'/C')\otimes \C$ is a Shimura curve of Mumford type.  Note in that theorem, a family that reaches the Arakelov bound is equivalent to a family with the maximal Higgs field. So we can apply that theorem in our situation. 
\end{proof}

\begin{appendices}
\section{The dot tensor product $G\odot H$} \label{the dot product}
In this appendix, let $G$ and $H$ be BT groups over any connected smooth variety $X$.  And $H$ is further an etale BT group with height h.

\begin{lemma} \label{tensor}
For any two integers $n<m$,
$G[p^n] \otimes_{\Z/p^m} H[p^m]=G[p^n]\otimes_{\Z/p^n} H[p^n]$.
\end{lemma}

\begin{proof}
Consider the following sequence: \[H[p^m] \xa{p^n} H[p^m] \ra H[p^n] \ra 0.\]
Tensoring with $G[p^n]$ yields 
\[H[p^m]\otimes_{\Z/p^m} G[p^n] \xa{p^n \otimes G[p^n]=0} H[p^m]\otimes_{\Z/p^m} G[p^n] \ra H[p^n] \otimes_{\Z/p^m} G[p^n] \ra 0\] is exact. Therefore $H[p^m]\otimes_{\Z/p^m} G[p^n] \cong H[p^n]\otimes_{\Z/p^m} G[p^n] \cong H[p^n]\otimes_{\Z/p^n} G[p^n].$
\end{proof}

\begin{lemma} \label{BT group}
$ G[p^n] \otimes_{\Z} H[p^n]$ is a BT$_n$ group scheme over $X$.
\end{lemma}

\begin{proof}
 Since $H[p^n]$ is etale, there exists an etale morphism $f_n: X' \ra X$ such that 
\[f^*_n( G[p^n]\otimes_\Z H[p^n]) \cong f^*_n(G[p^n])^{\times h}\] 
is a BT$_n$ group. Furthermore, it defines a descent datum with respect to the etale covering $X' \ra X$. Since $G[p^n]$ is a finite group scheme, $f^*_n( G[p^n])^{\times h}\ra X'$ is an affine morphism. By (\cite[ Chapter Descent, Lemma 33.1, Lemma 35.1]{Stacks Project}), the group scheme representing $f^*_n( G[p^n])^{\times h}$ descents to $C$ and represents $G[p^n]\otimes_\Z H[p^n]$.  Hence $G[p^n] \otimes_{\Z} H[p^n]$ is a group scheme. Obviously it is $p^n-$torsion. 

Furthermore, all the extra structure of BT$_n$ group can also descent from $f^*_n( G[p^n])^{\times h}$ to $G[p^n] \otimes_{\Z} H[p^n]$. Hence $G[p^n] \otimes_{\Z} H[p^n]$ is a BT group. 
\end{proof}

\begin{prop} \label{strong BT group}
$G \odot H$ is a BT group over $X$.
\end{prop}

\begin{proof}
Let 
\[i_n:G[p^n]\ra G[p^{n+1}] \]
 be the inclusions of truncated BT groups. Since each $H[p^{n+1}]$ is a flat $\Z/p^{n+1}$ module,  $i_n $ induces an inclusion \[G[p^n]\otimes H[p^{n+1}] \hookrightarrow G[p^{n+1}]\otimes H[p^{n+1}] \] with cokernel 
$G[p]\otimes H[p^{n+1}]. $ 
From \ref{tensor}, 
\[G[p^n] \otimes H[p^n] \cong G[p^n]\otimes H[p^{n+1}],G[p]\otimes H[p^{n+1}]\cong G[p]\otimes H[p] .\] Hence we have the short exact sequence of truncated BT groups
\begin{equation*} 
0 \ra G[p^n] \otimes H[p^n] \hookrightarrow G[p^{n+1}]\otimes H[p^{n+1}] \ra G[p]\otimes H[p] \ra 0.\end{equation*} 
Similarly, one can show this sequence holds even replacing 1 by any positive integer $k$.

Hence $\{G[p^n]\otimes H[p^n], i_n \}$ forms a inductive system and taking the inductive limit gives a BT group. 
\end{proof}

\section{Uniqueness of the lifting of the polarization}

Let $(A, \l)$ be a polarized abelian variety over $k$ and $A_n$ be an abelian variety over $W_n$. 
\begin{prop} \label{uniqueness}
If there exists a lifting polarization $\l_n$ of $\l$ on $A_n$, then the lifting is unique.
\end{prop}
\begin{proof}
Let $\o_A\subset \DD(A)_k$ be the Hodge filtration associated to $A$. Note the lifting problem of $(A,\l)$ to $W_n$ is equivalent to the lifting problem of $(\DD(A[p^\infty]), <\,>, \o_A \subset \DD(A)_k )$, i.e. finding a totally isotropic submodule of $\DD(A[p^\infty])_{W_n}$ lifting $\o_A$.

By the assumption, regardless of the totally isotropic condition, the lifting of the Hodge filtration is fixed. So the lifting of the polarization, if exists, is merely given by restriction of the symplectic form $<\,>$ to $\DD(A[p^\infty])_{|{W_n}/\Z_p}$.  Obviously, such restriction is unique. 
\end{proof} 
\end{appendices}

\bibliographystyle{amsplain}
\bibliography{mybib}{}

\end{document}